\documentclass[11pt,twoside,a4paper,reqno]{amsart}

\usepackage{amsmath}
	\usepackage{amsthm}
\usepackage{amsfonts, amssymb}
\usepackage[all]{xy}
\usepackage{url}

\usepackage{latexsym}

\usepackage{graphicx}
\usepackage{graphics}
\usepackage[scriptsize,bf]{caption}
\usepackage{float}
\usepackage{enumerate}
\usepackage{verbatim}
\usepackage{color}
\usepackage{wrapfig}
\usepackage{subfig}

\theoremstyle{plain}
\newtheorem{lema}{\sc Lemma}

\newtheorem{prop}[lema]{\sc Proposition}
\newtheorem{teo}[lema]{\sc Theorem}

\newtheorem{coro}[lema]{\sc Corollary}

\newtheorem*{hall}{\sc Hall's Theorem}
\theoremstyle{definition}

\newtheorem{obs}[lema]{\sc Remark}

\makeatletter
\@namedef{subjclassname@2020}{%
  \textup{2020} Mathematics Subject Classification}
  
\newcommand{\multiline}[1]{%
  \begin{tabularx}{\dimexpr\linewidth-\ALG@thistlm}[t]{@{}X@{}}
    #1
  \end{tabularx}
}    
\makeatother

\begin{document}

\title[A note on collapsibility of acyclic 2-complexes]{A note on collapsibility of acyclic 2-complexes}


\author[N. A. Capitelli]{Nicol\'as A. Capitelli\\
\textit{\scriptsize
U\MakeLowercase{niversidad} N\MakeLowercase{acional de} L\MakeLowercase{uj\'an}, D\MakeLowercase{epartamento de} C\MakeLowercase{iencias} B\MakeLowercase{\'asicas}, A\MakeLowercase{rgentina.}}}

\thanks{\textit{E-mail address:} {\color{blue}ncapitelli@unlu.edu.ar}}



\thanks{\textit{This research was partially supported by CONICET and the Department of Basic Sciences, UNLu (CDD-CB 148/18).}}

\begin{abstract} We present a Morse-theoretic characterization of collapsibility for $2$-dimensio\-nal acyclic simplicial complexes by means of the values of normalized optimal combinatorial Morse functions.
\end{abstract}

\maketitle


Let $K$ be a finite connected simplicial complex and let $f:K\rightarrow\mathbb{R}$ be a combinatorial Morse function over $K$. Let $\mathcal{Z}_f$ be the set of all combinatorial Morse functions $g:K\rightarrow \mathbb{Z}_{\geq 0}$ equivalent to $f$; i.e. inducing the same gradient field ($\mathcal{Z}_f\neq \emptyset$ by the finiteness of $K$). The \emph{nor\-ma\-li\-za\-tion} of $f$ is the map $h_f:K\rightarrow\mathbb{Z}_{\geq 0}$ defined by $$h_f(\sigma)=\min_{g\in\mathcal{Z}_f}\{g(\sigma)\}.$$ The function $h_f$ 
is also a combinatorial Morse function equivalent to $f$ (see Proposition \ref{Prop:1} below). The purpose of this note is to give a characterization of collapsibility for $2$-di\-men\-sio\-nal acyclic simplicial complexes by means of the values of $h_f$. In what follows, we shall write $\sigma\prec \tau$ whenever $\sigma$ is an immediate face of $\tau$ (i.e. a proper face of maximal dimension).









\begin{prop}\label{Prop:1} The function $h_f$ is a combinatorial Morse function equivalent to $f$.\end{prop}

\begin{proof} It suffices to show that $f(\sigma)<f(\tau)$ if and only if $h_f(\sigma)<h_f(\tau)$ whenever $\sigma\prec\tau$ (see \cite[Theorem 3.1]{Vil}). Suppose $f(\sigma)<f(\tau)$. If $g\in\mathcal{Z}_f$ is such that $h_f(\tau)=g(\tau)$ then in particular $g(\sigma)<g(\tau)$ and hence $$h_f(\sigma)\leq g(\sigma)<g(\tau)=h_f(\tau).$$
If now $h_f(\sigma)<h_f(\tau)$, let $g\in\mathcal{Z}_f$ be such that $h_f(\sigma)=g(\sigma)$. Then $$g(\sigma)=h_f(\sigma)<h_f(\tau)\leq g(\tau).$$ Since $f$ is equivalent to $g$ then $f(\sigma)<f(\tau)$.
\end{proof}

\begin{lema} \label{Lemma:PropertiesOfh} The function $h_f$ satisfies:\begin{enumerate}
\item\label{(1)} $h_f(\sigma)\geq\dim(\sigma)$ for all $\sigma\in K$.
\item\label{(2)} $h_f(\sigma)=0$ if and only if $\sigma$ is a critical vertex for $f$.
\item\label{(3)} If $\sigma\prec\tau$ and $f(\sigma)\geq f(\tau)$ then $h_f(\sigma)=h_f(\tau)$.
\end{enumerate}\end{lema}

\begin{proof} 
%
%
By definition, $h_f(v)\geq \dim(v)$ for any vertex $v\in K$. Let $\dim(\sigma)\geq 1$. Since in this case $\sigma$ has at least two immediate faces there is a $\nu\prec\sigma$ such that $h_f(\nu)<h_f(\sigma)$ (see, e.g., \cite[Theorem 9.3]{For}). By an inductive argument we conclude that $h_f(\sigma)>h_f(\nu)\geq \dim(\nu)=\dim(\sigma)-1$. This proves Item (1).

Item (2) follows from item (1) and the fact that lowering the value of any critical vertex in a function $g\in \mathcal{Z}_f$ produces again a combinatorial Morse function equivalent to $f$.

To see (3) suppose otherwise and let $\sigma$ be the simplex of minimal dimension satisfying $h_f(\sigma)>h_f(\tau)$. Note that $h_f(\tau)>h_f(\eta)$ for every $\eta\prec\sigma$. Indeed, if $\sigma'$ is the other $\dim(\sigma)$-dimensional simplex containing $\eta$ as an immediate face then, by the choice of $\sigma$, we have $h_f(\eta)\leq h_f(\sigma')< h_f(\tau)$. In particular $$h_f(\sigma)-1\geq h_f(\tau) >h_f(\eta)$$ for every $\eta\prec\sigma$. Therefore, the function $$g(\nu)=\begin{cases}h_f(\nu)&\nu\neq \sigma\\ h_f(\sigma)-1&\nu=\sigma\end{cases}$$ is a combinatorial Morse function equivalent to $f$, thus contradicting the minimality of $h_f$.
\end{proof}

For a given combinatorial Morse function $f:K\rightarrow \mathbb{R}$ consider the number $$\mathfrak{N}(K,f):=\sum_{\sigma\in K}(-1)^{\dim(\sigma)}h_f(\sigma).$$
This definition is motivated by property (3) of Lemma \ref{Lemma:PropertiesOfh}, which in turn implies that the sum may be taken over the critical simplices alone. We have the following result.

\begin{prop}\label{Proposition:collapsible0} If $K$ is collapsible then there exists a combi\-na\-to\-rial Morse function $f:K\rightarrow\mathbb{R}$ such that $\mathfrak{N}(K,f)=0$.\end{prop}

\begin{proof} If $K$ is collapsible then there exists a combinatorial Morse function $f$ over $K$ with only one critical simplex, which must be a vertex $v$ (see e.g. \cite[Lemma 4.3]{For}). Therefore $\mathfrak{N}(K,f)=h_f(v)=0$, the last equality holding by property ($2$) of Lemma \ref{Lemma:PropertiesOfh}.\end{proof}

\noindent In the case of graphs, the other implication also holds.

\begin{prop}\label{Proposition:ThmForGraphs} A connected graph $G$ is collapsible if and only if there exists a combi\-na\-to\-rial Morse function $f:G\rightarrow\mathbb{R}$ such that $\mathfrak{N}(G,f)=0$.\end{prop}

\begin{proof} Let $f$ be a Morse function with $\mathfrak{N}(G,f)=0$. Write
$$0=\sum_{\substack{\text{critical}\\ \text{vertices}}}h_f(v)-\sum_{\substack{\text{critical}\\ \text{edges}}}h_f(e).$$
\noindent By Lemma \ref{Lemma:PropertiesOfh} the first sum is zero and the second sum is positive if there is a critical edge. We conclude that $f$ has no critical edges. Since $G$ is connected there must be only one critical vertex. Hence $G$ is homotopy equivalent to CW with only a $0$-cell and thus it is a tree.\end{proof}

\begin{figure}[h]
  \centering
    \includegraphics[scale=0.6]{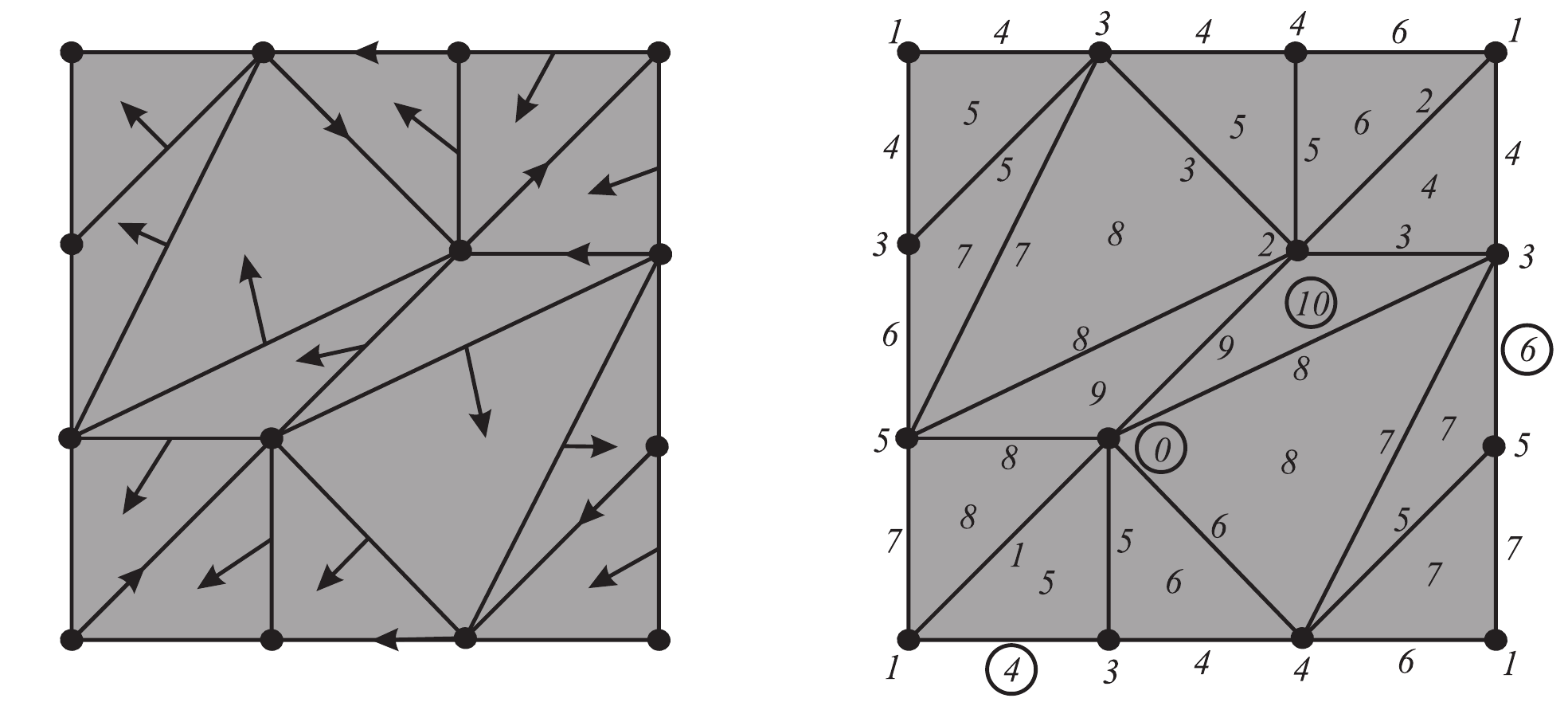}
        \caption{The gradient field (on the left) and the values of $h_f$ (on the right) for an optimal Morse function $f$ over a triangulation of the Torus for which $\mathfrak{N}(T,f)=0$ (the circled values correspond to critical simplices).}
\label{Figura:Unica}
\end{figure}

It is easy to see that Proposition \ref{Proposition:ThmForGraphs} does not hold in this generality for complexes of dimension greater than 1. Note however that the alleged functions appearing in these last two propositions can be taken to be \emph{optimal}; i.e. they have the least possible
number of critical simplices (among all combinatorial Morse functions over that complex). It is therefore natural to associate to a complex $K$ the number $$\mathfrak{N}(K):=\min\{|\mathfrak{N}(K,f)|\,:\, f:K\rightarrow\mathbb{R}\text{ optimal Morse function}\}.$$ With this definition, Proposition \ref{Proposition:collapsible0} may be restated as follows: ``If $K$ is collapsible then $\mathfrak{N}(K)=0$''. The converse of this statement does not hold in dimension greater than $1$ either (see Figure \ref{Figura:Unica}). However, the number $\mathfrak{N}$ can be used to characterize collapsibility for acyclic $2$-complexes. The main result of this note is the following.



\begin{teo}\label{Teo:acyclicintro} Let $K$ be an acyclic $2$-complex. Then, $K$ is collapsible if and only if\linebreak $\mathfrak{N}(K)=0$.\end{teo}

\noindent Before we prove Theorem \ref{Teo:acyclicintro} recall that, given a combinatorial Morse function $f:K\rightarrow\mathbb{R}$, the \emph{Morse complex associated to $f$} is the chain complex of $\mathbb{R}$-vector spaces
$$0\rightarrow\mathfrak{M}_k\stackrel{\partial_k}{\longrightarrow} \mathfrak{M}_{k-1}\stackrel{\partial_{k-1}}{\longrightarrow}\mathfrak{M}_{k-2}\stackrel{\partial_{k-2}}{\longrightarrow}\cdots,$$
where $\mathfrak{M}_k$ is the span of the critical $k$-simplices of $f$. By \cite[Theorem 8.2]{For}, this complex has the same homology with real coefficients as $K$. Also, \cite[Theorem 8.10]{For} shows that the boundary map $\partial_k:\mathfrak{M}_k\rightarrow\mathfrak{M}_{k-1}$ can be written $$\partial_k(\tau)=\sum_{\sigma\in\mathfrak{M}_{k-1}}\lambda^{\tau}_{\sigma}\sigma,$$ where the coefficients $\lambda^{\tau}_{\sigma}$ depend on the set $\Gamma(\tilde{\sigma},\sigma)$ of gradient paths between $\sigma$ and the immediate faces $\tilde{\sigma}$ of $\tau$ (see \cite[\S 8]{For}). In particular, if $\Gamma(\tilde{\sigma},\sigma)=\emptyset$ for every $\tilde{\sigma}\prec\tau$ then $\lambda^{\tau}_{\sigma}=0$.

We also shall make use of the following classical result from Graph Theory (see e.g. \cite{Wes}):

\begin{hall} A bipartite graph $G=(V,E)$ with partition $V=A\cup B$ admits a matching that saturates $A$ if and only if $|N(S)|\geq |S|$ for every $S\subset A$, where $N(S)$ denotes the set of vertices having a neighbor in $S$.\end{hall}

\begin{proof}[Proof of Theorem \ref{Teo:acyclicintro}] Let $L$ be a non-collapsible $2$-complex satisfying the hypotheses of the theorem. We shall show that $\mathfrak{N}(L)>0$. Let $f$ be an optimal combinatorial Morse function over $L$ and let $m_i(f)$ stand for the number of critical $i$-simplices of $f$. On one hand, $m_0(f)=1$ by \cite[Corollary 11.2]{For}. On the other hand, $m_1(f)=m_2(f)\geq 1$ by the \emph{weak Morse inequalities}  and the non-collapsibility of $L$ (see \cite[Corollary 3.7]{For} and \cite[Theorem 3.2]{Koz}). Let $A$ be the set of critical edges of $f$, $B$ the set of critical $2$-simplices of $f$ and form the (balanced) bipartite graph $G=(A\cup B,E)$, where we put an edge between $e\in A$ and $\sigma\in B$ if there exists a gradient path from an immediate face of $\sigma$ to $e$ (see \cite[\S 8]{For}). We claim that $G$ admits a complete matching (i.e. a matching involving every vertex of $G$). If this was not true, there exists by Hall's Theorem a subset $S\subset B$ such that $|S|>|N(S)|$, where $N(S)=\{e\in A\,|\,\{e,\sigma\}\in E\text{ for some $\sigma\in S$}\}$. Write $S=\{\sigma_1,\ldots,\sigma_r\}$. By the above remarks, $\{\partial_2(\sigma_1),\ldots,\partial_2(\sigma_r)\}\subset \mathsf{span}(N(S))$. Since $r>\dim(\mathsf{span}(N(S)))$ we can write $$0=\sum_{j=1}^rb_j\partial_2(\sigma_j),$$ for some $b_i\in\mathbb{R}$, not all zero. But in this case, $\sum_{j=1}^rb_j\sigma_j$ is a generating cycle of $H_2(\mathfrak{M}_{\ast},\partial_{\ast})\simeq H_2(L)$ and we reach a contradiction to our hypotheses. This proves that there exists a complete matching $\mathcal{M}$ in $G$. Order $A=\{e_1,\ldots,e_k\}$ and $B=\{\sigma_1,\ldots,\sigma_k\}$ so that $(e_i,\sigma_i)\in\mathcal{M}$ for every $i=1,\ldots,k$. By construction, there is a gradient path from a boundary edge of $\sigma_i$ to $e_i$ for every $i=1,\ldots,k$. In particular, $h_f(e_i)<h_f(\sigma_i)$ for every $i=1,\ldots,k$. We conclude that $$\mathfrak{N}(L,f)=-\sum_{j=1}^k h_f(e_j)+\sum_{j=1}^k h_f(\sigma_j)=\sum_{j=1}^k(h_f(\sigma_j)-h_f(e_j))>0.$$\end{proof}

\begin{obs}\label{obsgeneral} The hypotheses in the statement of the previous theorem can be slightly relaxed. The same proof can be carried out for connected $2$-complexes fulfilling $\chi(K)=1$ and $H_2(K)=0$. In particular, $\mathfrak{N}(\mathbb{R}P^2)>0$.\end{obs}

It is straightforward to produce similar results for $PL$-collapsibility. A complex is \emph{$PL$-collapsible} if it has a collapsible subdivision. For a complex $K$ one can define the number $$\widetilde{\mathfrak{N}}(K):=\min\{\mathfrak{N}(L)\,:\, L\text{ is a subdivision of }K\}.$$ As a direct corollary to Theorem \ref{Teo:acyclicintro} we have the following result.

\begin{coro} An acyclic $2$-complex $K$ is $PL$-collapsible  if and only if $\widetilde{\mathfrak{N}}(K)=0$.\end{coro}

\bigskip

\subsection*{\sc Acknowledgements} I would like to thank Gabriel Minian for many useful comments and suggestions.

\end{document}